\documentclass[12pt]{amsart}
\usepackage{amsmath,amsthm,latexsym,amscd,amsbsy,amssymb,url}
\usepackage[all]{xypic}
 \relax

\chardef\bslash=`\\ 

\makeatletter
\def\verbatim{\interlinepenalty\@M \@verbatim
  \leftskip\@totalleftmargin\advance\leftskip2pc
  \frenchspacing\@vobeyspaces \@xverbatim}
\makeatother
\newtheorem{thm}{Theorem}[section]
\newtheorem{cor}[thm]{Corollary}
\newtheorem{lem}[thm]{Lemma}
\newtheorem{pro}[thm]{Proposition}
\newtheorem{que}[thm]{Question}
\newtheorem{defin}[thm]{Definition}
\newtheorem{rem}[thm]{Remark}

\numberwithin{equation}{section}

\begin{document}


\title
{Periodic and Fixed Points of Multivalued maps on Euclidean spaces}
\author{R.~Z.~Buzyakova}
\address{Department of Mathematics and Statistics,
The University of North Carolina at Greensboro,
Greensboro, NC, 27402, USA}
\email{rzbouzia@uncg.edu}
\author{A.~Chigogidze}
\address{Department of Mathematics,
College of Staten Island,
Staten Island, NY, 10314, USA}
\email{alex.chigogidze@csi.cuny.edu}


\begin{abstract}{
We show, in particular, that a multivalued map $f$ from a closed
subspace $X$ of $\mathbb R^n$ to ${\rm exp}_k(\mathbb R^n)$ has a point of period exactly $M$
if and only if its continuous extension $\tilde f: \beta X\to {\rm exp}_k(\beta \mathbb R^n)$
has such a point. The result also holds if one repace $\mathbb R^n$ by a locally compact Lindel\"of space of finite dimension.
 We also show that if $f$ is a colorable map from
a normal space $X$ to the space ${\mathcal K}(X)$ of all compact subsets of $X$ then its extension $\tilde f:\beta X\to {\mathcal K}(\beta X)$ is fixed-point free.
}
\end{abstract}

\maketitle
\markboth{R.Z.Buzyakova, A. Chigogidze}{Periodic and Fixed Points of Multivalued maps on Euclidean spaces}
{ }

\section{Introduction}\label{S:intro}

Before we discuss the results of the paper we would like to give the definitions of the main concepts we 
study in this work.

By $\exp X$ we denote the space of all non-empty closed subsets of $X$ endowed with the Vietoris topology; $\exp_k X = \{A\in \exp X: |A|\leq k\}$;
and $\mathcal K (X)=\{K\in \exp X: K\ is \ compact\}$. Throughout the paper, $k$ is a positive integer in $\exp_k X$.
Let $X$ be a subset of $Z$ and  $f:X\to \exp Z$.  We say that $f$ fixes $x\in X$ if $x\in f(x)$. We say that $f$ has period $M$ at $x$ if $M$ is the smallest positive integer for which there exists a sequence $\langle x_1=x, x_2,...,x_M\rangle$ such that  $x_i\in f(x_{i-1})$ for $i=2,...,M$ and $x\in f(x_M)$.

It is proved in \cite[Theorem 2.10]{R2} that a continuous map $f$ from a closed subspace $X$ of $\mathbb R^n$ to $\exp_k (\mathbb R^n)$ has a fixed
point if and only if $\tilde f: \beta X\to \exp_k(\beta \mathbb R^n)$ has a fixed point. In this paper we show that this theorem is still true if one replace $\mathbb R^n$ by
any locally compact Lindel\"of space of finite dimension (Theorem \ref{thm:relativecolorableviaspectra}).  For our further discussion, recall that given a continuous map $f$
from a closed subspace $X$ of $Z$ into  $\exp Z$, a closed subset $F$ of $X$ is a color of $f$ 
if $F$ misses $\bigcup \{f(x):x\in F\}$. For a single-valued map $f:X\subset Z\to Z$, where $X$ is closed in $Z$,
this simply translates to the requirement that $F$ misses $f(F)$. Finally, $f$ is colorable if $X$ can be covered by finitely many colors.
The above mentioned theorem  \cite[Theorem 2.10]{R2} can be rewritten using coloring terminology as follows:
Any fixed-point free map from a closed subspace of $\mathbb R^n$ to $\exp_k \mathbb R^n$ is colorable. To a reader familiar with relations between fixed-point free
maps and colorability this may not seem surprising. The fact that a colorable self-map on a normal space $X$ has a fixed-point free extension
over $\beta X$ is a simple observation. When dealing with multivalued functions, however, 
significantly more work needs to be done to achieve an analogous statement.  
In case of single-valued map $f:X\to X$, it is an easy exercise to verify that given a coloring $\mathcal F$, the family $\mathcal G=\{f^{-1}(F):F\in \mathcal F\}$  is
a coloring as well. Moreover, thus defined coloring $\mathcal G$, has the property that $f(G)$ is closed for every $G\in \mathcal G$. This very property allows to conclude that every colorable single-valued self-map on a normal space has the fixed-point free extension over the  \v Cech-Stone compactification.
In section \ref{S:fixedpointfree} of this paper, we show that a colorable map $f$ from a normal space $X$ to its exponent
has a coloring with similar properties, namely, consisting of colors $F$ such that $F$ misses $cl_X(\bigcup \{f(x):x\in F\})$
(Lemma \ref{lem:goodcolors}). This result implies, in particular,
that  a colorable map from a normal space $X$ into the space of its compact subsets has the fixed-point free
extension $\tilde f:\beta X\to  \mathcal K( \beta X)$ (Corollary \ref{cor:colorabilityimpliesfpfextension}).

In Section \ref{S:periodicpoint}, we use the mentioned results about fixed-point free maps to derive a few results about periodic points.
We prove, in particular, that  a continuous map from a closed subspace $X$ of $\mathbb R^n$ to $\exp_k {\mathbb R}^n$ has a point of period
$M$ if and only if $\tilde f:\beta X\to \exp_k (\beta {\mathbb R}^n)$ has such a point. We then derive that a similar statement holds if
one replaces $\mathbb R^n$ with any locally compact Lindel\"of space of finite dimension. In the beginning of section  \ref{S:periodicpoint},
we give an outline of a quite straightforward argument of this fact for the case of single-valued self-maps and indicate complications that may occur
when one deals with single-valued not self-maps and even more complications in case of multivalued maps. For our results on periodic points it will be
 important that our fixed-point free map results are proved not only for self-maps but also for maps with smaller domains. We would like to
 remark that we find especially useful those statements in the fixed-point free map theory that deal with unequal range and domain (see, in particular,
 Remarks \ref{rem:importanceofrelativity} and \ref{rem:anotheruseofrelativity}).

Finally we would like to mention that colorability/fixed-point free map theory in Topology was inspired by works of Katetov \cite{K} and van Dowen \cite{D}.
A good account of results about colorability of some single-valued maps can be found in \cite{VM2}. 
 In notation and terminology, we will follow \cite{Eng}. 
For an arbitrary function $f: X\subset Z\to {\rm exp} Z$, by ${\rm Fix}(f)$ we denote the set of all points
$x$ for which $x\in f(x)$, that is, the set of all fixed points. Clearly the set of all
fixed points of a continuous map is closed in the domain. If $f$ is a contnuous
map from a closed subspace $X$ of a normal space $Z$ into ${\rm exp}_k(Z)$, symbol $\tilde f$
denotes the continuous map from $\beta X$ into ${\rm exp}_k(\beta Z)$ that coincides with $f$
on $X$. 

\section{More observations on fixed-point free multivalued maps}\label{S:fixedpointfree}

\par\bigskip
For convinience let us remind the definition of the Vietoris topology
on $\exp X$.
A standard neighborhood in $\exp X$ will be denoted as 
$$
\langle U_1,...,U_m\rangle = \{A\in \exp X: A\subset U_1\cup ...\cup U_m\ and\ U_i\cap A\not =\emptyset\ for \ all\ i=1,...,m\},
$$
where $U_1,...,U_m$ are open sets of $X$.

\par\bigskip
For desired results we will need a stronger versions of colorability, namely, colorability by bright colors.

\par\bigskip\noindent
\begin{defin}\label{defin:brightcolor}
For a continuous map $f$ from a closed subset $X$ of $Z$ into $\exp Z$, a closed  set $F\subset X$ is a bright color of $f$  if 
$F$ misses $cl_Z(\bigcup \{f(x):x\in F\})$.
\end{defin}

Our first result in this section is a generalization of the result \cite[Theorem 2.8]{R2} stating that every continuous
fixed-point free map $f$ from a closed subspace $X$ of $\mathbb R^n$ into $\exp_k(\mathbb R^n)$ is brightly colorable. 
For reference however, we will need the following obviously correct version of \cite[Theorem 2.8]{R2}:

\par\bigskip\noindent
\begin{thm}\label{thm:RnMainVersion} (version of \cite[Theorem 2.8]{R2})
Let $Y$ be a closed subspace of $\mathbb R^n$; $X$ a closed subspace of $Y$; and $f:X\to \exp_k Y$
a continuous fixed-point free map. Then $f$ is brightly colorable.
\end{thm}

\par\bigskip
We are now ready to prove our first result in this section. The terminology, basic facts, and references related to spectra
can be found in \cite{chibook}.  It should be noted that the result we are about to prove is a particular case of a more general statement 
about periodic points to be proved in the next section. However for readability purpose we choose to prove this theorem separately
and reference to its argument later when we prove the mentioned more general statement.

\par\bigskip\noindent
\begin{thm}\label{thm:relativecolorableviaspectra}
Let $Y$ be a locally compact Lindel\"of space of $\dim X = n$; $X$ its closed subspace;
and $f:X\to {\exp}_k(Y)$ a continuous fixed-point free map. Then $f$ is brightly colorable.
\end{thm}
\begin{proof}
By Theorem \ref{thm:RnMainVersion}, we may assume that 
$\omega(Y) > \omega$. By \cite{BC}, $Y = \lim{\mathcal S}_{Y}$, where ${\mathcal S}_{Y} = \{ Y_{\alpha}, q_{\alpha}^{\beta},A\}$ is a factorizing $\omega$-spectrum consisting of locally compact separable metrizable spaces $Y_{\alpha}$ and perfect projections $q_{\alpha}^{\beta} \colon Y_{\beta} \to Y_{\alpha}$, $\beta \geq \alpha$.

Since $X$ is closed in $Y$ it is the limit of the induced spectrum ${\mathcal S}_{X} = \{ X_{\alpha}, p_{\alpha}^{\beta}, A\}$, 
where $X_{\alpha} = q_{\alpha}(X)$ and $p_{\alpha}^{\beta} = q_{\alpha}^{\beta}|_{X_{\beta}}$, 
$\beta \geq \alpha$, $\alpha, \beta \in A$. Note that ${\mathcal S}_{X}$ is also an $\omega$-spectrum. 
It is factorizing since so is ${\mathcal S}_{Y}$ and $X$ is $C$-embedded in $Y$. Next consider the spectrum $\exp_{k}{\mathcal S}_{Y} = \{ \exp_{k}Y_{\alpha}, \exp_{k}q_{\alpha}^{\beta}, A \}$. Continuity of the functor $\exp_{k}$ guarantees that $\exp_{k}{\mathcal S}_{Y}$ is also a factorizing $\omega$-spectrum consisting of locally compact and Lindel\"{o}f spaces and perfect projections.  

(Note that if $\dim Y \leq n$, then, by \cite[Theorems 1.3.4 and 1.3.10]{chibook}, we may assume without loss of generality that each $Y_{\alpha}$ in the spectrum ${\mathcal S}_{Y}$ is also at most $n$-dimensional.) 

By the Spectral Theorem \cite[Theorem 1.3.4]{chibook} applied to the spectra ${\mathcal S}_{X}$, $\exp_{k}{\mathcal S}_{Y}$, and the map $f \colon X \to \exp_{k}Y$, 
we may assume (if necessary passing to a cofinal and $\omega$-complete subset of $A$) that for each $\alpha \in A$ there 
is a map $f_{\alpha} \colon X_{\alpha} \to \exp_{k}Y_{\alpha}$ such that $\exp_{k}q_{\alpha}\circ f = f_{\alpha}\circ p_{\alpha}$. Since $f$ is 
fixed-point free and $X$ is Lindel\"{o}f, we can find a countable functionally open cover $\{ G_{i} \colon i \in \omega\}$ of $X$ and 
a countable collection $\{ \mathcal U_{i} \colon i\in \omega\}$ of functionally open subsets of $\exp_k Y$ such that $f(G_{i}) \subset \mathcal U_{i}$ and 
$(\bigcup \mathcal U_{i})\cap G_{i} = \emptyset$ for each $i \in \omega$. Factorizability of spectra ${\mathcal S}_{X}$ and ${\mathcal S}_{Y}$ 
guarantees (see \cite[Proposition 1.3.1]{chibook}) the existence of an index $\alpha_{i} \in A$ such that 
$G_{i} = p_{\alpha_{i}}^{-1}(p_{\alpha_{i}}(G_{i}))$ and $\mathcal U_{i} = (\exp_k q_{\alpha_{i}})^{-1}(\exp_k q_{\alpha_{i}}(\mathcal U_{i}))$, $i \in \omega$. 
Choose $\beta \in A$ so that $\beta \geq \alpha_{i}$ for each $i \in \omega$  -- this is possible because $A$ 
is an $\omega$-complete set (see \cite[Corollary 1.1.28]{chibook}) --  and note that $G_{i} = p_{\beta}^{-1}(p_{\beta}(G_{i}))$ and 
$\mathcal U_{i} = (\exp_k q_{\beta})^{-1}(\exp_k q_{\beta}(\mathcal U_{i}))$ for each $i \in \omega$. Obviously, $p_{\beta}(G_{i}) \cap [\bigcup \exp_k q_{\beta}(\mathcal U_{i})] = \emptyset$, 
$i \in \omega$, from which it follows that $f_{\beta} \colon X_{\beta} \to \exp_{k}Y_{\beta}$ is fixed-point free.   

By Theorem \ref{thm:RnMainVersion}, $f_{\beta} \colon X_{\beta} \to \exp_{k}Y_{\beta}$ is brightly colorable. Let $\{ F_{j}\}$ be a finite closed cover of $X_{\beta}$ consisting of bright colors of $f_{\beta}$. In order to complete the proof it suffices to show that each of the sets $p_{\beta}^{-1}(F_{j})$ is a bright color with respect to $f$. Indeed, assuming that this is not the case we can find $x \in p_{\beta}^{-1}(F_{j}) \cap \operatorname{cl}_{Y}(\bigcup\{ f(x) \colon x \in p_{\beta}^{-1}(F_{j})\})$. Then $p_{\beta}(x) \in F_{j}$ and 
\begin{multline*}
p_{\beta}(x) = q_{\beta}(x) \in q_{\beta}\large( \operatorname{cl}_{Y}(\bigcup\{ f(x) \colon x \in p_{\beta}^{-1}(F_{j})\})\large) \subset\\
 \operatorname{cl}_{Y_{\beta}}(q_{\beta}( \bigcup\{ f(x) \colon x \in p_{\beta}^{-1}(F_{j}\})) = 
 \operatorname{cl}_{Y_{\beta}}( \bigcup\{ q_{\beta}(f(x)) \colon x \in p_{\beta}^{-1}(F_{j}\})) = \\
\operatorname{cl}_{Y_{\beta}}( \bigcup\{ f_{\beta}(p_{\beta}(x)) \colon x \in p_{\beta}^{-1}(F_{j}\})) \subset \operatorname{cl}_{Y_{\beta}} (\bigcup\{ f_{\beta}(z) \colon z \in F_{j}\}),
\end{multline*}
\noindent which contradicts brightness of $F_{j}$.
\end{proof}

\par\bigskip\noindent
For our further discussion we need the following statement proved in \cite{R2}, in which ${\mathcal K}(X)$
is the subspace $\{K\in \exp X: K\ is \ compact\}$ of $\exp X$.
\par\bigskip\noindent
\begin{pro}\label{pro:proposition29}(\cite[Proposition 2.9]{R2})
If $X$ is a closed subspace of a normal space $Z$; $f:X\to {\mathcal K}(Z)$ continuous; and $\mathcal F$ a bright
coloring of $f$, then  the family $\{\beta F:F\in {\mathcal F}\}$ is a bright coloring of $\tilde f:\beta X\to {\mathcal K}(\beta Z)$.
\end{pro}
\par\bigskip\noindent
For our next statement we will use the following particular case of Proposition \ref{pro:proposition29}:
{\it If $X$ is a closed subspace of a normal space $Z$, $f:X\to \exp_k Z$ continuous, and $\mathcal F$ a bright
coloring of $f$, then  the family $\{\beta F:F\in {\mathcal F}\}$ is a bright coloring of $\tilde f:\beta X\to \exp_k(\beta Z)$.} 
Thus, Proposition \cite[Proposition 2.9]{R2}  and our Theorem \ref{thm:relativecolorableviaspectra} imply the following:

\par\bigskip\noindent
\begin{cor}\label{cor:relativecriteriaspectra}
Let $Z$ be a locally compact Lindel\"of space of finite dimension; $X$ its closed subspace;
and $f:X\to {\exp}_k(Z)$ a continuous map. Then $f$ is fixed-point free
if and only $\tilde f: \beta X\to {\exp}_k(\beta Z)$ is fixed-point free.
\end{cor}

\par\bigskip\noindent
In the next section when dealing with periodic points we will need the following proposition
which is an easy consequence of the above statement.
\par\bigskip\noindent
\begin{pro}\label{pro:closureofFix}
Let $f:X\to {\rm exp}_k Z$ be 
a continuous map, $Z$ a locally compact  Lindel\"of space of finite dimension, and $X$  closed in $Z$.
Then ${\rm Fix}(\tilde f) =cl_{\beta X}({\rm Fix}(f))$.
\end{pro}
\begin{proof}
The inclusion "$\supset$" is obvious. To prove "$\subset$",
pick any $x\not \in cl_{\beta X}({\rm Fix} (f))$. We need to show that $\tilde f$ does not fix $x$.
For this, let $O_x$ be an open neighborhood of $x$ in $\beta Z$ whose
closure misses $cl_{\beta X}({\rm Fix} (f))$.
Put $Y = cl_{\beta Z}(O_x)\cap X$. Then $f|_Y: Y\to {\exp}_k(Z)$
is fixed-point free. By Corollary \ref{cor:relativecriteriaspectra}, $\tilde f|_{\beta Y}$ is fixed-point free as well.
Hence, $\tilde f$ does not fix $x$. 
\end{proof}

\par\bigskip\noindent
\begin{rem}\label{rem:importanceofrelativity}
Observe that in the proof of Proposition \ref{pro:closureofFix} we need the full version 
of Corollary \ref{cor:relativecriteriaspectra} (that is, including the case of unequal range and domain) even for the case of the proposition when $X=Z$.
\end{rem}

\par\bigskip
In general, colorability of $f:X\to X$ does not apply that $\tilde f: \beta X\to \beta X$
is fixed-point free. If $X$ is normal, however, then the implication is true and is an easy observation
from the fact that if $\mathcal F$ is a coloring of $f$, then so is $\{f^{-1}(F):F\in {\mathcal F}\}$.
For multivalued maps more work is needed for this conclusion and is done in the remaining part of this section.

\par\bigskip\noindent
\begin{lem}\label{lem:goodcolors}
Let $X$ be normal and $f:X\to \exp X$  continuous. If $\mathcal F$ is an $n$-sized coloring of $f$ then
there exists an at most $2^n$-sized bright coloring of $f$. 
\end{lem}
\begin{proof} We will construct our coloring inductively.
Let ${\mathcal H}_i$, where $0\leq i<n$, be the set of all subcollections
of $\mathcal F$ of size at least  $(n-i)$.
\par\medskip\noindent
{\it Step 0}: Put $O_0=\emptyset $. 
\par\medskip\noindent 
{\it Assumption}: Assume that  an open set $O_{k-1}\subset X$, where $k-1<n$,  is 
	defined and the following hold:
	\begin{description}
		\item[\rm A1] $\bigcap {\mathcal G}\subset O_{k-1}$ for every ${\mathcal G}\in {\mathcal H}_{k-1}$;
		\item[\rm A2] $cl_X(O_{k-1})$ can be covered by at most $|{\mathcal H}_{k-1}|$ many bright colors.
	\end{description}
	\par\smallskip\noindent
	Before we perform our induction step let us show that A1 and A2 hold for $i=0$. Since 
	${\mathcal H}_0=\{\mathcal F\}$ we only need to  show that
	$\bigcap {\mathcal F} =\emptyset$. For this fix $x\in X$ and $y\in f(x)$. Since $\mathcal F$ is a cover
	there exists $F_y\in {\mathcal F}$ that contains $y$. It suffices to show now that $F_y$ does not contain $x$.
	Since $F_y$ is a color, it misses
	$\bigcup \{f(z):z\in F_y\}$. 
	Since $F_y\cap f(x)\not =\emptyset$ we conclude that $f(x)\not\subset \bigcup\{f(z):z\in F_y\}$, whence
	$x\not\in F_y$. 
	
\par\medskip\noindent 
{\it Step $k< n$}: Let ${\mathcal H}_k'$ be the set of all $(n-k)$-sized subsets of ${\mathcal F}$ such that the following
	hold:
	\par\medskip\noindent
	\begin{description}
		\item [\rm P] $\bigcap {\mathcal G}\setminus O_{k-1}\not = \emptyset$ for every ${\mathcal G}\in {\mathcal H}_k'$.
	\end{description}
	\par\medskip\noindent
	To define $O_k$ let us prove three claims first.
\par\bigskip\noindent
{\it Claim 1.} {\it Let ${\mathcal G}\in {\mathcal H}'_k$. Then 
		$f(\bigcap {\mathcal G})\subset \{S\in \exp X: S\subset \bigcup ({\mathcal F}\setminus {\mathcal G})\}$.}
		
		\par\smallskip\noindent
		 By the definition of color, $f(G)$ misses $\bigcup\{f(x):x\in G\}$
		for every $G\in {\mathcal G}$. Hence $f(G)\subset \{S\in \exp X: S\ misses\ G\}$.
		Therefore, $f(\bigcap {\mathcal G})\subset \{S\in \exp X: S\ misses\ \bigcup {\mathcal G}\}$.
		The last set is a subset of $\{S\in \exp X: S\subset \bigcup ({\mathcal F}\setminus {\mathcal G})\}$, which proves the claim.
		
\par\bigskip\noindent
{\it Claim 2.} {\it Let ${\mathcal G}\in {\mathcal H}'_k$. Then $[\bigcap {\mathcal G}\setminus O_{k-1}]$ misses $\bigcup ({\mathcal F}\setminus {\mathcal G})$.}

\par\smallskip\noindent
 Indeed, if 
		$[\bigcap {\mathcal G}\setminus O_{k-1}]$ had a common point $p$ with some $F\in {\mathcal F}\setminus {\mathcal G}$ then
		${\mathcal G}\cup \{F\}$ would have been in ${\mathcal H}_{k-1}$. Therefore, $p$ would have 
		been in $O_{k-1}$, contradicting $p\in \bigcap {\mathcal G}\setminus O_{k-1}$.
		
\par\bigskip\noindent
{\it Claim 3.} {\it Let ${\mathcal G}\in {\mathcal H}'_k$. Then there exists an open neighborhood $U_{\mathcal G}$ of $[\bigcap {\mathcal G}\setminus O_{k-1}]$ 
		whose closure is a bright color.} 

\par\smallskip\noindent
		By Claim 1, we have 
		$\bigcup\{f(x):x\in[\bigcap {\mathcal G}\setminus O_{k-1}]\}\subset \bigcup ({\mathcal F}\setminus{\mathcal G})$.
		By Claim 2, $[\bigcap {\mathcal G}\setminus O_{k-1}]$ misses $\bigcup ({\mathcal F}\setminus{\mathcal G})$.
		The set $\bigcup ({\mathcal F}\setminus{\mathcal G})$ is closed as the union of finitely many closed sets.
		Therefore, $[\bigcap {\mathcal G}\setminus O_{k-1}]$ misses $cl_X(\bigcup ({\mathcal F}\setminus{\mathcal G}))$
		and consequently $cl_X(\bigcup\{f(x):x\in[\bigcap {\mathcal G}\setminus O_{k-1}]\})$. Therefore,
		$[\bigcap {\mathcal G}\setminus O_{k-1}]$  is a bright color. Proposition \cite[Proposition 2.2]{R2} states that due to normality
		of $X$ any bright color can be placed
		in an open neighborhood whose closure is a bright color as well. Thus,
		a desired  $U_G$ exists, which proves the claim. 
 
\par\bigskip\noindent
	Put $O_k = O_{k-1}\cup [\bigcup \{U_{\mathcal G}:{\mathcal G}\in {\mathcal H}_k'\}]$. Clearly A1 and A2 are satisfied. The inductive
	construction is complete
	
	\par\bigskip\noindent
	The family ${\mathcal F}' = \{\overline U_{\mathcal G}: {\mathcal G}\in {\mathcal H}_k, 0\leq k<n\}$ is a desired bright coloring.

\end{proof}

\par\bigskip\noindent
\begin{cor}\label{cor:colorabilityimpliesfpfextension}
Let $X$ be normal and let $f:X\to {\mathcal K}(X)$ be a colorable continuous map. Then the continuous extension
$\tilde f:\beta X\to {\mathcal K}(\beta X)$ is fixed point free.
\end{cor}
\begin{proof}
The conclusion  follows from Proposition \ref{pro:proposition29} and Lemma \ref{lem:goodcolors}.
\end{proof}

\par\bigskip
In connection with Corollary \ref{cor:colorabilityimpliesfpfextension}, we would like to remark that 
the authors do not know an answer to the following related question:

\par\bigskip\noindent
\begin{que}\label{que:relativenormal}
Is it true that  colorability of $f:X\to {\mathcal K}(Z)$ implies colorability
of $\tilde f:\beta X\to {\mathcal K}({\beta Z})$, where $Z$ is normal and  $X$ is closed in $Z$. 
What if ${\mathcal K}(Z)$ is replaced by $\exp_k Z$? What if ${\mathcal K}(Z)$ is replaced by $Z$?
\end{que}

\par\bigskip\noindent
Observe that if $Z$ is a Lindel\"of locally compact space of finite dimension then the answer is affirmative for the second part of the question
and follows from Theorem \ref{thm:relativecolorableviaspectra}

\par\bigskip\noindent
\section{Periodic points of multivalued maps to Euclidean hyperspaces}\label{S:periodicpoint}

\par\bigskip
In this section we study periodic points of multi-valued maps. First let us remind the definition we use:

\par\bigskip\noindent
\begin{defin}\label{defin:periodatpoint} 
Let $f:X\subset Z\to \exp Z$ be a map and $M$ a positive integer. We say that $f$ has period 
$M$ at $x\in X$ if $M$ is the smallest positive integer such that
there exists a sequence $\langle x_1=x, x_2,..,x_M \rangle$
with the property that $x_{i+1}\in f(x_i)$ for all $1\leq i<M$ and $x\in f(x_M)$.
\end{defin}

\par\bigskip
Before we  dive into multivalued case in our full generaity we would like to present a short and quite transparent argument for
a very specific case of our main result. We will then discuss how this argument can and will be woven into the 
proofs of this section. The specific case we would like to show first is that a continuous map $f:\mathbb R^5\to \mathbb R^5$
has a point of period $3$ if its continuous extension $\tilde f:\beta \mathbb R^5\to \beta\mathbb R^5$ has a point of period $3$.
To prove this, fix $p\in \beta \mathbb R^5$  at which $\tilde f$ has period $3$. We may assume that $p$ is in the remainder of the compactification.
We have neither $\tilde f$ nor $\tilde f\circ \tilde f$ fixes $p$. Therefore we can find an open neighborhood $U$ of $p$ in $\beta \mathbb R^5$
whose closure in $\beta \mathbb R^5$ mises the fixed points of $\tilde f$ and those of $\tilde f\circ \tilde f$. Put $A=cl_{\beta\mathbb R^5}(U)\cap \mathbb R^5$.
We have $(\tilde f\circ\tilde f\circ\tilde f)|_{\beta A}$  fixes $p$. Here we need Theorem \cite[Theorem 3.5]{BC}
stating, in particular, that a continuous map $f$ from a closed subspace $X$ of $\mathbb R^n$ is fixed-point free if and only if $\tilde f:\beta X\to \beta\mathbb R^n$
is fixed-point free. Using this theorem and the fact that 
$(\tilde f)^3|_{\beta A} = \widetilde {(f^3|_A)}$
we conclude that
$f^3|_A$ fixes some $x\in A$. Since, by our choice,  $A$ does not contain any points fixed by $f$ or $f\circ f$ we conclude that
$f$ has period $3$ at $x$, which proves our particular case. 

The convenience of the described situation is that $f^2, f^3$ are defined on all of $\mathbb R^5$ because $f$ is a selfmap. If one wishes to
prove a similar statement with function going from a closed subspace of $\mathbb R^5$ then one immediately encounters a sitation
when $f^2$ may not be defined on a part of the original domain.
If one deals with a function going to $\exp_2 \mathbb R^5$ instead, then $f^2$ (to be defined later) maps $\mathbb R^5$
into $\exp_4 \mathbb R^5$, that is, goes outside of the original range. If one deals both with a smaller domain and hyperspace as the range one faces a bouquet of  problems that need to be
taken care of. We will overcome the difficulties in two stages. First we prove our result for $\mathbb R^n$ and then derive the desired conclusion
for any locally-compact Lindel\"of space of finite dimension. But first we would like to make a remark of advertising nature:

\par\bigskip\noindent
\begin{rem}\label{rem:anotheruseofrelativity}
Observe that in our simple example of selfmap on $\mathbb R^5$ we eventually reached the situation when we had to use the fixed-point free theorem for non-selfmap case, namely, for the case $A\subset \mathbb R^5\to \mathbb R^5$.
\end{rem}

\par\bigskip
To upgrade the just presented argument for a more general case, we need the following 
concept of the multivalued map theory:

\par\bigskip\noindent
\begin{defin}\label{defin:composition}
Let $f: X\subset Z\to \exp Z$ be a map. Define 
$f^1(x) = f(x)$ and $f^{n+1}(x) = \bigcup \{f(y):y\in f^n(x)\}$, provided
the right side is a closed non-empty subset of $Z$.
\end{defin}
\par\bigskip
Clearly, if $f$ maps $X$ into $\exp_k X$ then $f^n$ is defined on all of $X$ (recall we agreed that
$k$ is a positive integer in $\exp_k X$).
The following is a standard fact of the multivalued map theory, which we prove here for the sake
of completeness.

\par\bigskip\noindent
\begin{pro}\label{pro:continuouscomposition} (Folklore) 
Let 
$f: X\to{\exp}_k X$ be continuous. Then $f^n$ is a continuous map defined on the entire $X$ with
range in ${\exp}_{k^n} X$.
\end{pro}
\begin{proof}
Assume $f^{n-1}$ is continuous. To prove that $f^n$ is continuous, fix $x\in X$ and an open neighborhood
$\mathcal W$ of $f^n(x)$ in ${\rm exp} X$. By the definition of $f^n$, we have 
$f^n(x) = \bigcup \{f(y):y\in f^{n-1}(x)\}$. Since $|f^{n}(x))|$ is finite, we may assume that
${\mathcal W}=\langle W_{z}: z\in f^{n}(x)\rangle$, where $W_{z}$ is an open neighborhood
of $z$ in $X$. By continuity of $f$, for each $y\in f^{n-1}(x)$, we can find open $V_y$ in $X$
such that $f(V_y)\subset \langle W_{z}: z\in f(y)\rangle$. Put ${\mathcal V}= \langle V_y:y\in f^{n-1}(x) \rangle$.
Clearly, $\mathcal V$ is an open neighborhood of $f^{n-1}(x)$ in ${\rm exp } X$.
By continuity of $f^{n-1}$, we can find an open neighborhood $U$ of $x$ such that
$f^{n-1}(U)\subset {\mathcal V}$. It is clear that 
$f^n(U)\subset {\mathcal W}$. 
\end{proof}
\par\bigskip\noindent
The next statement is also a standard fact and is corollary to the fact that the maps whose equality
is to be proved
coincide on $X$, which is dense in $\beta X$. Recall, we agreed that given $f:X\subset Z\to \exp_k Z$,
by $\tilde f$ we denote the continuous map from $\beta X$ to $\exp_k \beta Z$ that coincides with
$f$ on $X$.
\par\bigskip\noindent
\begin{pro}\label{pro:commute} 
Let $g: X\to {\rm exp}_k  X$ be continuous.
Then $\widetilde {(g^m)}=\tilde g^m$.
\end{pro}

\par\bigskip
The next statement is a self-obvious fact and we will state it  without a proof and
use it without reference. 

\par\bigskip\noindent
\begin{rem}\label{rem:mperiodicpointfixedbyfm}
Let $f:X\subset Z\to {\rm exp} Z$ be a map and $M$ a positive integer. Then $f$ has period $M$ at $x\in X$
if and only if $f^M$ fixes $x$ and $f^K$ does not for $K<M$.
\end{rem}

\par\bigskip
Our next statement allows us to reduce the case "$X\subset \mathbb R^n$" to
"$X= \mathbb R^n$". To prove the statement we need the theorem of Jaworowski that if a metric compactum $C$ is an absolute retract then so
is $\exp_k C$ (a proof can also be found in \cite{F}). Since $\exp_k \mathbb R^n$ embeds into $\exp_k [0,1]^n$ as an open subspace,
$\exp_k \mathbb R^n$ is an absolute extensor - the property we will use.

\par\bigskip\noindent
\begin{pro}\label{pro:goodextension}
Let $X$ be a closed subspace of $\mathbb R^{n-1}\times \{0\}$ and
$f:X\to {\rm exp}_k(\mathbb R^{n-1}\times \{0\})$ a continuous map.
Then there exists a continuous extension $g:\mathbb R^n\to {\exp}_k \mathbb R^n$ of $f$
such that the following hold:
\begin{enumerate}
	\item $\tilde g(y)$ misses $\beta X$ if $y\not\in \beta X$; and
	\item For $x\in \beta X$ and a positive integer $M$, the map $\tilde f$ has period $M$ at $x$ 
	if and only if $\tilde g$ has period $M$ at $x$.
\end{enumerate}
\end{pro}
\begin{proof}
Let $h$ be a continuous extension of $f$ over $\mathbb R^{n-1}\times \{0\}$ to
${\exp}_k(\mathbb R^{n-1}\times \{0\})$. This extension exists since $\exp_k (\mathbb R^n\times \{0\})$ is an absolute
extensor. Define $g:\mathbb R^n\to {\rm exp}_k(\mathbb R^n)$
as follows: 
$$
g(\langle x(1),...,x(n)\rangle) =\{\langle y(1),...,y(n-1), x(n)+dist(\pi_{\mathbb R^{n-1}}(x), X) \rangle :
$$
$$
\langle y(1),...,y(n-1), 0\rangle\in h(\langle x(1),...,x(n-1),0\rangle)\},
$$
where $\pi_{\mathbb R^{n-1}}$  denotes the projection of $\mathbb R^{n-1}\times \mathbb R$ onto $\mathbb R^{n-1}$-axis.
In words, $g$ coincides with $h$, and consequently with $f$, on $X$. At other points it first  acts as $h$
relative to the $\mathbb R^{n-1}$-axis and then slightly moves the "$h$"-image of each point  along the $\mathbb R$-axis (the n-th axis in $\mathbb R^n$), 
which guarantees part
1 of the conclusion of the proposition and continuity.

To prove part 2 of the conclusion, fix $x\in \beta X$. Since $\tilde g$ extends
$\tilde f$,  the period of $\tilde g$ at $x$ cannot exceed
that of $\tilde f$ at $x$. Next we assume that $\tilde g$ has period $M$ at $x$ and we need to show
that the period of $\tilde f$ at $x$ is $M$ as well. For this let
$\langle x_1, ..., x_M\rangle$ be a sequence of elements of $\beta\mathbb R^n$ witnessing the periodicity $M$
of $\tilde g$ at $x$. By part 1, if at least one element of the sequence falls outside of $\beta X$ so do
the rest of the elements and consequently $x$ cannot be in $\tilde g(x_M)$. Therefore, all elements
of the sequence are in $\beta X$, which demonstrates the desired conclusion.
\end{proof}

\par\bigskip\noindent
\begin{thm}\label{thm:periodicpoint}
Let $f:X\to {\exp}_k \mathbb R^n$ be 
a continuous map, where $X$ is closed in $\mathbb R^n$.
Then $f$ has a point of period $M$ if and only if $\tilde f:\beta X\to {\ exp}_k \mathbb R^n$ has.
\end{thm}
\begin{proof}
Necessity is obvious. To prove sufficiency we may assume that $n$-th coordinate of each
point of $X$ is $0$ and the range of $f$ 
is ${\exp}_k \mathbb ( R^{n-1}\times \{0\})$ (this assumption can always be achieved by appropriately placing
the domain space into a higher-dimensional  Euclidean space). Fix $z\in \beta X$ at which $\tilde f$
has period $M$. We may assume that $z\in \beta X\setminus X$. 
We have $\tilde f^m$ does not fix $z$ for $m<M$.
Let $g$ be a continuous extension of $f$ that satisfies the conclusion of Proposition \ref{pro:goodextension}.
By part 2 of Proposition \ref{pro:goodextension},
$\tilde g^m$ does not fix $z$ either if $m<M$. Therefore, there exists an open neighborhood $O_z$ of $z$ in $\beta \mathbb R^n$
whose closure does not meet ${\rm Fix}(\tilde g^m)$ for every $m<M$.
By Proposition \ref{pro:commute}, $cl_{\beta \mathbb R^n}(O_z)$ misses ${\rm Fix}(\widetilde {(g^m)})$ for
every $m<M$. Therefore, $cl_{\beta\mathbb R^n}(O_z)$ misses ${\rm Fix}(g^m)$
for every $m<M$.
Put $Y = cl_{\beta\mathbb R^n}(O_z)\cap \mathbb R^n$. Since $\tilde g^M =\widetilde {(g^M)}$ (Proposition \ref{pro:commute}),
we conclude that $\widetilde {(g^M)}|_{\beta Y}$ fixes $z$. Therefore, by Proposition \ref{pro:closureofFix},
$z\in cl_{\beta\mathbb R^n}({\rm Fix}(g^M|_Y))$. Select any
$y\in {\rm Fix}(g^M|_Y)$. By the choice of $O_z$, we have $y\not \in {\rm Fix}(g^m)$ for every $m<M$.
Therefore, $g$ has period $M$ at $y$. Since $y\in X$, applying part 2 of Proposition \ref{pro:goodextension},
we conclude that $f$ has period $M$ at $y$.
\end{proof}

\par\bigskip
Next we will show that the previous result holds if we replace $\mathbb R^n$ by a locally compact
Lindel\"of space of finite dimension. Our proof will be derived from $\mathbb R^n$-version of the statement.
To shorten our arguments we introduce the following concept:

\par\bigskip\noindent
\begin{defin}\label{defin:Nbrightcolor}
A color $F$ of $f:X\subset Z\to \exp_k Z$ that misses
$cl_Z(\bigcup f^n(F))$ for all $n\leq N$ is an  $N$-bright color of $f$.
\end{defin}

\par\bigskip\noindent
\begin{lem}\label{lem:verygoodneighborhood}
Let $Z$ be normal; $X$ a closed subspace of $Z$; and $f:X\to \exp_k Z$ a continuous map without
points of period less than or equal to $N$. Then for any $x\in X$ one can find an open neighborhood
$U$ of $x$ in $Z$ such that $cl_Z(U)$ is an $N$-bright color of $f$.
\end{lem}
\begin{proof}
For simplicity, put $N=2$. For any $z\in  f^2(x)$  fix an open neighborhood
$U_z$ of $z$ in $Z$ whose closure misses $x$. This is possible because $x\not\in f^2(x)$ by hypothesis.
Now for any $y\in  f(x)$, fix an open neighborhood
$U_y$ of $y$  with the following properties:
\begin{enumerate}
	\item The closure of $U_y$ misses $x$; and 
	\item $f(U_y)\subset \langle U_z: z\in f(y)\rangle$.
\end{enumerate}
The first requirement is achievable due to the fact that $f$'s period at $x$ is greater than $2$.
The second requirement is achievable due to continuity of $f$.
Finally, we can find an open neighborhood $U$ of $x$ with the following properties:
\begin{enumerate}
	\item $cl_Z(U)$ misses $cl_Z(U_y)$ for every $y\in f(x)\cup f^2(x)$; and
	\item $f(U)\subset \langle U_y:y\in f(x)\rangle$.
\end{enumerate}
The first property is achievable because both $f(x)$ and $f^2(x)$ are finite and $x\not\in cl_Z(U_y)$ for all
$y\in f(x)\cup f^2(x)$ by our choice. The second property is achievable due to continuity of $f$.
Clearly, this $U$ is as desired.
\end{proof}

\par\bigskip\noindent
\begin{lem}\label{lem:betaofNbrighcolor}
Let $Z$ be a normal space; $X$ its closed subspace; and $f:X\to \exp_k Z$ a continuous map.
If $\mathcal F$ is an $N$-bright coloring of $f$, then $\mathcal G = \{\beta F:F\in \mathcal F\}$
is an $N$-bright coloring of $\tilde f$.
\end{lem}
\begin{proof}
Since $X$ is normal and $\mathcal F$ is a cover of $X$, we conclude that $\mathcal G$ is a cover of $\beta X$.
It is left to show that each element of $\mathcal G$ is an $N$-bright color of $f$. Fix $\beta F\in \mathcal G$.
We have
$$
cl_{\beta Z}(\bigcup \tilde f^n(\beta F))= cl_{\beta Z}(\bigcup \{\tilde f^n(x):x\in \beta F\})=cl_{\beta Z}(\bigcup \{f^n(x):x\in  F\}).
$$
Observe that the rightmost part of this equality misses $\beta F$ because $Z$ is normal and 
$cl_{Z}(\bigcup \{f^n(x):x\in  F\})$ misses $F$ by hypothesis.
\end{proof}

\par\bigskip\noindent
\begin{thm}\label{thm:periodiclocalycompact}
Let $Y$ be a locally compact Lindel\"of space of finite dimension; $X$ its closed subspace; and $f: X\to \exp_k Y$ continuous.
Suppose $\tilde f: \beta X\to \exp_k \beta Y$ has a point of period $M$. Then $f$ has a point of
period $M$.
\end{thm}
\begin{proof}
We assume that $f$ has no points of period $M$ and we need to show that $\tilde f$ has no such points either.
By theorem \ref{thm:periodicpoint}, we may assume that $Y$ is not metrizable.
As in the argument of Theorem \ref{thm:relativecolorableviaspectra},  we represent $Y$ as the inverse limit of a factorizing $\omega$-spectrum 
${\mathcal S}_{Y} = \{ Y_{\alpha}, q_{\alpha}^{\beta},A\}$ consisting of locally compact separable metrizable spaces $Y_{\alpha}$ of the same dimension as $Y$ and perfect projections $q_{\alpha}^{\beta} \colon Y_{\beta} \to Y_{\alpha}$, $\beta \geq \alpha$.
Similarly, ${\mathcal S}_{X} = \{ X_{\alpha}, p_{\alpha}^{\beta}, A\}$, 
where $X_{\alpha} = q_{\alpha}(X)$ and $p_{\alpha}^{\beta} = q_{\alpha}^{\beta}|_{X_{\beta}}$, 
$\beta \geq \alpha$, $\alpha, \beta \in A$; and $\exp_{k}{\mathcal S}_{Y} = \{ \exp_{k}Y_{\alpha}, \exp_{k}q_{\alpha}^{\beta}, A \}$. 
Also as in the argument of  Theorem \ref{thm:relativecolorableviaspectra} we may assume that for each $\alpha\in A$, there
is a map $f_{\alpha} \colon X_{\alpha} \to \exp_{k}Y_{\alpha}$ such that $\exp_{k}q_{\alpha}\circ f= f_{\alpha}\circ p_{\alpha}$.

Since $f$ has no points of period at most $M$, by Lemma \ref{lem:verygoodneighborhood}, we can find family $\{F_n:n\in \omega\}$
consisting of  functionally closed sets that are $M$-bright colors for $f$.
 As in the proof of Theorem \ref{thm:relativecolorableviaspectra}, we can find $\alpha^*\in A$ such that
 for every element $\beta\in A$ greater than $\alpha^*$, the set $p_\beta(F_n)$ is an $M$-color for $f_\alpha$.
 Fix $\alpha>\alpha^*$ in $A$.
 By Theorem \ref{thm:periodicpoint}, we conclude that $\tilde f_\alpha : \beta X_\alpha\to \exp_k \beta Z_\alpha$ has
 no points of period at most $M$.
 By Lemma \ref{lem:verygoodneighborhood}, $\tilde f_\alpha$ has an $M$-bright coloring. Therefore, $f_\alpha$ has an $M$-bright
 coloring  $\mathcal F$.
 As in the argument of Theorem  \ref{thm:relativecolorableviaspectra}, we conclude that $\mathcal S= \{p^{-1}_\alpha (F): F\in \mathcal F\}$
 is a bright coloring of $f$. By Lemma \ref{lem:betaofNbrighcolor}, $\mathcal G = \{\beta S:S\in \mathcal S\}$
 is an $M$-bright coloring of $\tilde f$, whence $\tilde f$ has no points of period $M$.
\end{proof}
\par\bigskip\noindent
\begin{thm}\label{thm:main}
Let $Z$ be a locally compact Lindel\"of space of finite dimension and $X$ its closed subspace. Then a continuous map $f: X\to\exp_k Z$
has a point of period $M$ if and only if $\tilde f:\beta X\to \exp_k \beta Z$ has.
\end{thm}

\par\bigskip\noindent
In our final statement  $P(g, K)$ is the set of all points of the domain of $g$ at which $g$ has period at most $K$.

\par\bigskip\noindent
\begin{thm}\label{thm:closureof periodicpoints}
Let $X$ be a closed subspace of a locally compact Lindel\"of space $Z$ of finite dimension.
If $f:X\to \exp_k Z$ is continuous then $P(\tilde f, M)=cl_{\beta X}(P(f, M))$ for any positive integer $M$.
\end{thm} 
\begin{proof}
The $\supset$-inclusion is obvious. To prove the inclusion "$\subset$", fix $x\not\in cl_{\beta X} (P(f, M))$ and an open neighborhood $U$ of $x$
in $\beta X$ whose closure misses $cl_{\beta X} (P(f, M))$. Put $A= cl_{\beta X}(U)\cap X$. Then $f|_A$ has no
points of period less than or equal to $M$. By Theorem \ref{thm:periodiclocalycompact}, $\widetilde {(f|_A)}$ has no such points either.
Since $\widetilde {(f|_A)}$ coincides with $\tilde f|_A$ and $x\in A$ we conclude that 
$x\not\in P(\tilde f, M)$, which proves the statement.
\end{proof}


\begin{thebibliography}{99}



\bibitem{R2}
R.~Z.~Buzyakova, {\it Multivalued Fixed-point free maps on Euclidean spaces}, Proc. AMS, accepted.

\bibitem{BC}
R.~Z.~Buzyakova, A.~Chigogidze, 
{\it Fixed-point free maps of Euclidean spaces}, Fundam Mathematicae {\bf 212} (2011), 1--16.

\bibitem{chibook}
A.~Chigogidze, {\it Inverse Spectra}, North Holland, Amsterdam, 1996.

\bibitem{D}
E.~van Douwen, {\it $\beta X$ and fixed-point free maps}, Topology Appl. {\bf 51} (1993), 191--195.

\bibitem{Eng}
R.~Engelking, {\it General Topology}, PWN, Warszawa, 1977.

\bibitem{F}
V.~V.~Fedor\v cuk, {\it Covariant functors in a category of compacta, absolute retracts and Q-manifolds}, (Russian) 
Uspekhi Mat. Nauk 36 (1981), no. 3(219), 177Ð195, 256.

\bibitem{K}
M.~Kat\v{e}tov, {\it A theorem on mappings}, Comm. Math. Univ. Carolinae {\bf 8} (1967), 431--433.

\bibitem{VM2}
J. van Mill, {\it The infinite-dimensional topology of function spaces}, Elsevier, Amsterdam, 2001.
\end{thebibliography}
\end{document}